\documentclass[12pt,reqno]{amsart}

\usepackage{amssymb,latexsym}

\usepackage{enumerate}

\usepackage[french,english]{babel}
\usepackage{amsmath}
\usepackage{graphicx}
\usepackage{amssymb}
\usepackage{bbm}
\usepackage{amsthm,mathtools}
\usepackage{ulem}
\usepackage{geometry}
\usepackage{tikz-cd}
\usepackage{mathrsfs}
\usepackage[colorinlistoftodos]{todonotes}
\usepackage{enumitem}
\usepackage{verbatim}
\usepackage[foot]{amsaddr}
\usepackage{dsfont}

\makeatletter

\@namedef{subjclassname@2010}{
	
	\textup{2020} Mathematics Subject Classification}

\makeatother
\newtheorem{thm}{Theorem}[section]
\newtheorem*{thm*}{Theorem}

\newtheorem{lem}[thm]{Lemma}

\theoremstyle{definition}

\numberwithin{equation}{section}

\newcommand{\mbr}{\mathbb{R}}

\newcommand{\mcf}{\mathcal{F}}

\newcommand{\mcl}{\mathcal{L}}

\newcommand{\mcm}{\mathcal{M}}

\newcommand{\mme}{\mathrm{e}}

\newcommand{\M}{{\mathcal M}}
\newcommand{\D}{{\mathcal D}}

\newcommand{\F}{\mathcal F}

\usepackage{hyperref}
\hypersetup{hypertex=true,colorlinks=true,linkcolor=blue,anchorcolor=blue,citecolor=blue}
\usepackage{color}
\newcommand{\newabstract}[1]{%
	\par\bigskip
	\csname otherlanguage*\endcsname{#1}%
	\csname captions#1\endcsname
	\item[\hskip\labelsep\scshape\abstractname.]
}

\begin{document}

	\baselineskip=17pt

	\title[Large  quadratic character sums with multiplicative coefficients]{Large  quadratic character sums with multiplicative coefficients}

	\author{Zikang Dong\textsuperscript{1}}
		\author{Yutong Song\textsuperscript{2}}
        \author{Weijia Wang\textsuperscript{3}}
            \author{Hao Zhang\textsuperscript{4}}
    \author{Shengbo Zhao\textsuperscript{2}}
	\address{1.School of Mathematical Sciences, Soochow University, Suzhou 215006, P. R. China}
	\address{2.School of Mathematical Sciences, Key Laboratory of Intelligent Computing and Applications(Ministry of Education), Tongji University, Shanghai 200092, P. R. China}
    \address{3.Morningside Center of Mathematics, Academy of Mathematics and Systems Science, Chinese Academy of
		Sciences, Beijing 100190, P. R. China}
	\address{4.School of Mathematics, Hunan University, Changsha 410082, P. R. China}
	\email{zikangdong@gmail.com}
    \email{99yutongsong@gmail.com}
	\email{weijiawang@amss.ac.cn}
	\email{zhanghaomath@hnu.edu.cn}
	\email{shengbozhao@hotmail.com}

	\date{}
	
	\begin{abstract} 
		In this article, we investigate conditional large values of quadratic Dirichlet character sums with multiplicative coefficients. We prove some Omega results  under the assumption of the generalized Riemann hypothesis.
	\end{abstract}

	
	\maketitle
	
	\section{Introduction}
  For the lower bounds for the maximal size of character sums:
  $$\max_{\chi\neq\chi_0({\rm mod}\;q)}\Big|\sum_{n\le x}\chi(n)\Big|,$$
  Granville and Soundararajan's celebrated work \cite{GS01} showed several important results. In recent years, quite a few researchers have worked on improving their results.
    When $x=\exp\left(\tau\sqrt{\log q\log_2q}\right)$, Hough \cite{Hou} showed that
   \[ \max_{\chi\neq\chi_0({\rm mod}\;q)}\Big|\sum_{n\le x}\chi(n)\Big|\ge \sqrt{x}\exp\bigg((1+o(1))A(\tau+\tau')\sqrt{\frac{\log X}{\log_2 X}}\bigg),\]
 where $A,\tau,\tau'\in\mbr$ such that $\tau=(\log_2q)^{O(1)}$ and
   \[\tau=\int_A^\infty \frac{\mme^{-u}}{u}d u,\qquad \tau'=\int_A^\infty\frac{\mme^{-u}}{u^2}d u.\]
   When $\exp((\log q)^{\frac12+\delta})\le x\le q$, La Bret\`eche and Tenenbaum \cite {BT} showed that
$$ \max_{\chi\neq\chi_0({\rm mod}\;q)}\Big|\sum_{n\le x}\chi(n)\Big|\ge \sqrt x\exp\bigg((\sqrt2+o(1))\sqrt{\frac{\log( q/x)\log_3(q/x)}{\log_2(q/x)}}\bigg).$$
The above two results was generalized by the authors \cite{DSWZZ} to the sums with multiplicative coefficients:
$$ \max_{\chi\neq\chi_0({\rm mod}\;q)}\Big|\sum_{n\le x}f(n)\chi(n)\Big|,$$
where $f(\cdot)$ is a completely multiplicative function satisfying some additional conditions. Before that, one could only show a lower bound of the size $\sqrt x\exp((c+o(1))\sqrt{\log q/\log_2q})$, though for a larger set of $f(\cdot)$. See \cite{DLSZ}.

The aim of this paper is to study the lower bounds for the maximum size of the quadratic character sums  with multiplicative coefficients:
$$\max_{X<|d|\le2X\atop d\in\D}\Big|\sum_{n\le x}f(n)\chi_d(n)\Big|,$$where $\D$ denotes the set of all fundamental discriminants.
We define a subset of the completely multiplicative functions:
\begin{equation*}
    \mcf \coloneqq \big\{f\;{\rm completely\; multiplicative}:|f(n)|=1, \forall n \in \mathbb{N};\, {\rm{Re}}f(n)\overline{f(m})\ge 0,\;\forall \,m,n \big\}.
\end{equation*}

When $x$ is around $\exp((\log X)^{\frac12})$, we have the following result, which generalize Theorem 1.2 of \cite{DZ}.
\begin{thm}\label{thm1.2}
    Assume GRH. Let $\exp\left(4\sqrt{\log X\log_2X}\log_3X\right)\le x\le \exp((\log X)^{\frac12+\varepsilon})$.
Then we have 
   $$\max_{X<|d|\le 2X\atop d\in\D}\Big|\sum_{n\le x}f(n)\chi_d(n)\Big|\ge \sqrt{x}\exp\bigg((\tfrac{\sqrt2}{2}+o(1))\sqrt{\frac{\log X}{\log_2 X}}\bigg)$$
   for all $f \in \mcf $.
\end{thm}

When the sum is long, we have the result generalizing Theorem 1.3 of \cite{DZ}. 
\begin{thm}\label{thm1.3}
    Assume GRH. Let $\exp((\log X)^{\frac12+\varepsilon})<x\le   X^{\frac12}$. Then we have 
    $$ \max_{X<|d|\le 2X\atop d\in\D}\Big|\sum_{n\le x}f(n)\chi_d(n)\Big|\ge \sqrt x\exp\bigg((1+o(1))\sqrt{\frac{\log(\sqrt X/x)\log_3(\sqrt X/x)}{\log_2(\sqrt X/x)}}\bigg)$$
       for all $f \in \mcf $.
\end{thm}
 Note that this also improves the previous result \cite{DLSZquadra}, at the cost of an additional condition on $f(\cdot)$.

 We refer to Lamzouri's work \cite{Lam24} for the distribution of large quadratic character sums and the work of part of the authors \cite{DWZ23} for the structure.
    \section{Preliminary Lemmas}\label{sec2}
   
Firstly, we need the following conditional estimate for the mean values of quadratic characters. This improves the previous unconditional result of Granville and Soundararajan \cite[Lemma 4.1]{GS} a lot.
    \begin{lem}\label{lem2.2}
	Assuming GRH. Let $n=n_0n_1^2$ be a positive integer with $n_0$ the square-free part of $n$.
	 Then for any $\varepsilon>0$, we obtain
	\begin{align*}
	\sum_{|d|\le X\atop d\in\D} \chi_{d}(n)=\frac{X}{\zeta(2)}\prod_{p|n}\frac{p}{p+1}\mathds{1}_{n=\square}+ O\left(X^{\frac12+\varepsilon}g_1(n_0)g_2(n_1)\right),
	\end{align*}
	where   ${\mathds{1}}_{n=\square}$ indicates the indicator function of the square numbers, and
    $$g_1(n_0)=\exp((\log n_0)^{1-\varepsilon}),\;\;\;\;g_2(n_1)=\sum_{d|n_1}\frac{\mu(d)^2}{d^{\frac12+\varepsilon}}.$$
\end{lem}
\begin{proof}
    This is Lemma 1 of \cite{DM}.
\end{proof}
It is clear that 
$$g_1(n_0)\le n_0^\varepsilon\le n^\varepsilon,\;\;\;\;g_2(n_1)\le n_1^\varepsilon\le n^\varepsilon.$$
\par
The following lemma serves as a key ingredient in proving Theorem \ref{thm1.2}.
\begin{lem}\label{rmrn}    Let $y$ be large and $\lambda=\sqrt{\log y\log_2 y}$. Define the multiplicative function $r(\cdot)$ supported on square-free integers: for any prime $p$:$$r(p)=\begin{cases}   \frac{\lambda}{\sqrt p \log p}, &  \lambda^2\le p\le exp((\log\lambda)^2),\\   0, & {\rm otherwise.}\end{cases}$$
If $\log N>3\lambda\log_2\lambda$, then we have
\begin{equation}\label{DD}    \sum_{k,\ell\le N}\sum_{m,n\le  Y\atop mk=n\ell}r(a)r(b)\Big/\sum_{n\le Y}r(n)^2\ge N\exp{\bigg((2+o(1))\sqrt{\frac{\log y}{\log_2 y}}}\bigg).\end{equation}\end{lem}
\begin{proof}    This follows directly from \cite[p. 97]{Hou}.\end{proof}
The following result for GCD sums plays a key role in the proof of Theorem \ref{thm1.3}.
\begin{lem}\label{GCD}
    Let $\M$ be any set of positive square-free integers with $|\M|=N$. Then as $N\to\infty$, we have
    $$\max_{|\M|=N}\sum_{m,n\in\M}\sqrt{\frac{(m,n)}{[m,n]}}=N\exp\bigg((2+o(1))\sqrt{\frac{\log N\log_3N}{\log_2N}}\bigg).$$
\end{lem}
\begin{proof}
    This is \cite[Eq. (1.5)]{BT}.
\end{proof}



\section{Proof of Theorem \ref{thm1.2}}\label{sec4}

The idea for the proof of Theorem \ref{thm1.2} follows from Hough's work \cite{Hou}, which is originally Soundararajan's resonance method \cite{Sound}.  Let $y=X^{\frac12-\alpha}/x^2$ and  $\lambda=\sqrt{\log y\log_2y}$, where $0<\alpha<1/4$ is any fixed small number. We define the function $r(n)$ as in Lemma \ref{rmrn}. Then, define 
$$R_d:= R_d(f) = \sum_{n\leq y}f(n)r(n)\chi_d(n).$$
We consider the following two sums
$$M_1(R,X):=\sum_{X<|d|\le 2X\atop d\in\D}|R_d|^2,$$
and
$$M_2(R,X):=\sum_{X<|d|\le 2X\atop d\in\D}|S_d(x)|^2|R_d|^2,$$
where $S_d(x) \coloneqq S_d(f,x) =\sum_{n\le x}f(n)\chi_d(n).$ Obviously,
\begin{equation}
    \label{max1}\max_{X<|d|\le 2X\atop d\in\D}|S_d(x)|^2\ge\frac{M_2(R,X)}{M_1(R,X)}.
\end{equation}
\par
For $M_1(R,X)$, we have
$$
M_1(R,X) = \sum_{m,n \le y}f(m)\overline{f(n)}r(m)r(n)\sum_{X<|d|\le 2X\atop d\in\D}\chi_d(mn).
$$
Using Lemma \ref{lem2.2} and the fact that $r(n)$ is supported on the set of square-free numbers, we get
\begin{align}\label{M1upper}
M_1(R,X)&=\frac{X}{\zeta(2)}\sum_{m\le y}r(m)^2\prod_{p|m}\frac{p}{p+1}+O\Big(X^{\frac12+\varepsilon}y^\varepsilon\sum_{m,n\le y}r(m)r(n)\Big)  \nonumber \\
&\le\frac{X}{\zeta(2)}\sum_{m\le y}r(m)^2+O\Big(X^{\frac12+\varepsilon}y\sum_{m\le y}r(m)^2\Big)\nonumber\\
&=\frac{X}{\zeta(2)}\sum_{m\le y}r(m)^2+O\Big(X^{1-\alpha+\varepsilon}\sum_{m\le y}r(m)^2\Big).
\end{align}
Here we used the Cauchy-Schwarz inequality, the fact that $\prod_{p \mid m}\frac{p}{p+1} \le 1$, and finally $y=X^{\frac12-\alpha}/x^2 \ll X^{\frac12-\alpha}$.
\par
For $M_2(R,X)$, we employ Lemma \ref{lem2.2} and obtain
\begin{align}\label{M2}
    M_2(R,X) = \, & \sum_{a,b \le x}\sum_{m,n \le y}f(m)\overline{f(n)}f(a)\overline{f(b)}r(m)r(n)\sum_{X<|d|\le 2X\atop d\in\D}\chi_d(abmn) \nonumber \\
    = \, & \frac{X}{\zeta(2)}\sum_{a,b\le x}\sum_{m,n \le y \atop abmn=\square}f(am)\overline{f(bn)}r(m)r(n)\prod_{p \mid abmn}\frac{p}{p+1} \nonumber \\
    & +  O\Big(X^{\frac{1}{2}+\varepsilon}x^\varepsilon y^\varepsilon \sum_{a,b \le x}\sum_{m,n \le y}r(m)r(n) \Big) \nonumber \\
    = \, & \frac{X}{\zeta(2)}\sum_{a,b\le x}\sum_{m,n \le y \atop abmn=\square}f(am)\overline{f(bn)}r(m)r(n)\prod_{p \mid abmn}\frac{p}{p+1} \nonumber \\
    & +  O\Big(X^{\frac{1}{2}+\varepsilon}x^{2+\varepsilon} y^{1+\varepsilon} \sum_{m\le y}r(m)^2 \Big). 
\end{align}
We use $\frac{X}{\zeta(2)}\mathcal{L}$ to denote the main term of $M_2(R,X)$, then we have
\begin{align*}
    &\mathcal{L} \\= \, & \sum_{a,b\le x}\sum_{m,n \le y \atop am=bn}r(m)r(n)\prod_{p \mid am}\frac{p}{p+1} + 2{\rm{Re}}\sum_{a,b\le x}\sum_{m,n \le y \atop am> bn} f(am)\overline{f(bn)}r(m)r(n)\prod_{p \mid abmn}\frac{p}{p+1} \\
    \ge \, & \sum_{a,b\le x}\sum_{m,n \le y \atop am=bn}r(m)r(n)\prod_{p \mid am}\frac{p}{p+1}.
\end{align*}
Here we utilized the non-negativity of $r(n)$, the positivity of the inner product, and $f \in \F$. By the prime number theorem, 
\begin{equation}\label{hlower}
    \prod_{p \mid am}\frac{p}{p+1} = \exp \Big( \sum_{p \mid am} \log\Big(1-\frac{1}{p+1}\Big)\Big) \ge \exp\Big(- \sum_{p \le X}\frac{1}{p} + O\Big(\sum_{p\le X}\frac{1}{p^2} \Big) \Big) \ge (\log X)^{-\delta}
\end{equation}
for some absolute constant $\delta>0$. Thus, 
\begin{equation*}
    \mcl \ge (\log X)^{-\delta}\sum_{a,b\le x}\sum_{m,n \le y \atop am=bn}r(m)r(n).
\end{equation*}
Back to \eqref{M2}, we have the following lower bound 
\begin{align}
       \label{M2lower}
    M_2(R,X) \ge \, & \frac{X}{\zeta(2)}(\log X)^{-\delta}\sum_{a,b\le x}\sum_{m,n \le y \atop am=bn}r(m)r(n) +  O\Big(X^{\frac{1}{2}+\varepsilon}x^{2+\varepsilon} y^{1+\varepsilon} \sum_{m\le y}r(m)^2 \Big) \nonumber \\
    = \, & \frac{X}{\zeta(2)}(\log X)^{-\delta}\sum_{a,b\le x}\sum_{m,n \le y \atop am=bn}r(m)r(n) + O\Big(X^{1-\alpha+\varepsilon} \sum_{m\le y}r(m)^2 \Big) ,
\end{align}
since $y=X^{\frac12-\alpha}/x^2$. 
Combining \eqref{M1upper} and \eqref{M2lower}, we establish the following lower bound for the ratio of $M_2(R,X) $ and $M_1(R,X) $
\begin{equation*}
    \frac{M_2(R,X)}{M_1(R,X)} \ge (\log X)^{-\delta}\sum_{a,b\le x}\sum_{ m,n\leq y\atop am=bn}r(m)r(n)\Big/\sum_{m\le y}r(m)^2+O(X^{-\alpha+\varepsilon}).
\end{equation*}
Returning to \eqref{max1}, we complete the proof of Theorem \ref{thm1.2} with the help of Lemma \ref{rmrn}.

\section{Proof of Theorem \ref{thm1.3}}\label{sec5}

Let $\mcm$ be a set of positive square-free integers that satisfies $y_\M=\max_{m\in\M}P_+(m)\le(\log N)^{1+o(1)}$, where $P_+(m)$ denotes the largest prime factor of $m$. Meanwhile, for any small constant $\kappa < 1/2$, we let $|\M|=N=\lfloor X^{\frac12-\kappa}/x\rfloor$.
\par
Then, we define the resonator as
$$
R_d \coloneqq R_d(f) =\sum_{m\in\M} f(m)\chi_d(m).
$$
Let
$$M_1(R,X) \coloneqq \sum_{X<|d|\le 2X\atop d\in\D}|R_d|^2,$$
and 
$$M_2(R,X) \coloneqq \sum_{X<|d|\le 2X\atop d\in\D}|S_d(x)|^2|R_d|^2,$$
where $S_d(x)$ is defined in the same way as in Section \ref{sec4}.
Plainly, we have 
\begin{equation}
    \label{max}\max_{X<|d|\le 2X\atop d\in\D}|S_d(x)|^2\ge\frac{M_2(R,X)}{M_1(R,X)}.
\end{equation}

\par
Next, we deal with $M_1(R,X)$ and $M_2(R,X)$ separately, aiming to obtain their effective bounds. For $M_1(R,X)$, substituting the definition of $R_d$ into it and changing the order of summation, we obtain that
$$
M_1(R,X) = \sum_{m,n \in \M}f(m)\overline{f(n)}\sum_{X<|d|\le 2X\atop d\in\D}\chi_d(mn).
$$
For $m,n \in \M$, $mn=\square$ yields $m=n$. So we split the sum into two parts: the case $m=n$ and the case $m\neq n$ and get
$$
M_1(R,X) = \sum_{m \in \M}\sum_{X<|d|\le 2X\atop d\in\D}\chi_d(m^2) + \sum_{m,n \in \M \atop m \neq n}f(m)\overline{f(n)}\sum_{X<|d|\le 2X\atop d\in\D}\chi_d(mn).
$$
Here we use the fact that $|f(n)|=1$. Employing Lemma \ref{lem2.2}, we have
\begin{align*}
    M_1(R,X) =\, & \frac{X}{\zeta(2)}\sum_{m \in \M}\prod_{p \mid m}\frac{p}{p+1} + O(X^{\frac{1}{2}+\varepsilon}N^\varepsilon\sum_{m,n \in \M \atop m \neq n}1) \\
    \le \, & \frac{X}{\zeta(2)}N + O(X^{\frac{1}{2}+\varepsilon}N^{2+\varepsilon}).
\end{align*}
Since $N \le  X^{\frac12-\kappa}$, we obtain the following upper bound for $M_1(R,X)$
\begin{equation}\label{M1upperbound}
    M_1(R,X) \le (1+o(1))\frac{X}{\zeta(2)}N.
\end{equation}
\par
Furthermore, for $M_2(R,X)$, we have
$$
M_2(R,X) = \sum_{m,n \in \M}\sum_{a,b \le x}f(a)\overline{f(b)}f(m)\overline{f(n)}\sum_{X<|d|\le 2X\atop d\in\D}\chi_d(abmn).
$$
Using Lemma \ref{lem2.2} again, we divide the above sum into two parts by considering $abmn= \square$ and $abmn \neq \square$ and get
\begin{align*}
   & M_2(R,X) \\=\, & \frac{X}{\zeta(2)}\sum_{m,n \in \M}\sum_{a,b \le x \atop abmn = \square}f(am)\overline{f(bn)}\prod_{p \mid abmn}\frac{p}{p+1} + O(X^{\frac{1}{2}+\varepsilon}N^\varepsilon x^\varepsilon\sum_{m,n \in \M}\sum_{a,b \le x}1) \\
    =\, & \frac{X}{\zeta(2)}\sum_{m,n \in \M}\sum_{a,b \le x \atop abmn = \square}f(am)\overline{f(bn)}\prod_{p \mid abmn}\frac{p}{p+1} + O(X^{\frac{1}{2}+\varepsilon}N^{2}x^{2}).    
\end{align*}
By \eqref{hlower}, we have
\begin{align*}
    M_2(R,X) \ge \, & \frac{X}{\zeta(2)}(\log X)^{-\delta}\sum_{m,n \in \M}\sum_{a,b \le x \atop abmn = \square}f(am)\overline{f(bn)} + O(X^{\frac{1}{2}+\varepsilon}N^{2}x^{2})
\end{align*} for some absolute constant $\delta>0$.
Since $f \in \mcf$, we have
$$
{\rm{Re}}\sum_{m,n \in \M}\sum_{a,b \le x \atop am \neq bn}f(am)\overline{f(bn)} >0.
$$
Thus, we have 
$$ M_2(R,X) \ge  \frac{X}{\zeta(2)}(\log X)^{-\delta}\sum_{m,n \in \M}\sum_{a,b \le x \atop am=bn}1+ O(X^{\frac{1}{2}+\varepsilon}N^{2}x^{2}).$$
Combining with \eqref{M1upperbound} and using $N= X^{\frac12-\kappa}/x$ we have
$$
\frac{M_2(R,X)}{M_1(R,X)} \gg \frac{1}{N}(\log X)^{-\delta}\sum_{m,n \in \M}\sum_{a,b \le x \atop am = bn}1 + O(X^{-\kappa+\varepsilon}x).
$$
\par
Then we focus on the inner sum. For fixed $m,n$, $am=bn$ yields $a=nA/(m,n)$ and $b=mA/(m,n)$, where $A$ is an integer. Subsequently, we have
$$
\sum_{a,b\le x \atop am=bn}1 \ge \frac{x}{\max\{\frac{m}{(m,n)},\frac{n}{(m,n)}\}}.
$$
Since $\max \M \le 2 \min \M$, we get
$$
\sum_{a,b\le x \atop am=bn}1 \ge \frac{x}{\sqrt{2\frac{m}{(m,n)}\frac{n}{(m,n)}}} \gg x\sqrt{\frac{(m,n)}{[m,n]}}.
$$
Then,
\begin{align}\label{GCDsums}
    \sum_{m,n \in \M}\sum_{a,b \le x \atop am = bn}1 \gg \, &x\sum_{m,n\in\M\atop{{[m,n]}}/(m,n)\le x^2/2}{\sqrt{\frac{(m,n)}{[m,n]}}}  \nonumber \\
    = \, & x\bigg(\sum_{m,n\in\M
}\sqrt{\frac{(m,n)}{[m,n]}}-\sum_{m,n\in\M
\atop [m,n]/(m,n)> x^2/2}\sqrt{\frac{(m,n)}{[m,n]}}\bigg).
\end{align}
According to \cite[p. 25]{BT}, we use Rankin's trick to deal with the last sum above. For $\eta>0$ which will be chosen later, we have
\begin{align*}
    \sum_{m,n\in\M \atop [m,n]/(m,n)> x^2/2}\sqrt{\frac{(m,n)}{[m,n]}} \ll x^{-2\eta}\sum_{m,n\in\M
}\bigg({\frac{(m,n)}{[m,n]}}\bigg)^{\frac12-\eta}.
\end{align*}
Fix $m$, 
\begin{align*}
    \sum_{n\in\M}\bigg({\frac{(m,n)}{[m,n]}}\bigg)^{\frac12-\eta} \le \prod_{p\le y_\M}\bigg(1+\frac{2}{p^{\frac12-\eta}-1}\bigg) \ll \exp\big( y_\M^{\frac12+\eta}\big).
\end{align*}
Noting that $y_\M=\max_{m\in\M} P_+(m)\le(\log (X^{\frac12-\kappa}/x))^{1+o(1)}$ and $x>\exp((\log X)^{\frac12+\delta}),$ we choose $\eta = \kappa/3$ and obtain
\begin{align*}
    \sum_{m,n\in\M \atop [m,n]/(m,n)> x^2/2}\sqrt{\frac{(m,n)}{[m,n]}} \ll \, & x^{-2\eta} \sum_{m \in \M}\exp\big( y_\M^{\frac12+\eta}\big) \\
    \ll \, & x^{-2\eta}N\exp\big( (\log (X^{\frac12-\kappa}/x))^{\frac12+\eta+o(1)}\big) \\
    \ll \, & N \exp\big(-\tfrac23\kappa(\log X)^{\frac12+\kappa}\big)\exp\big((\log X)^{\frac12+\frac{2}{3}\kappa}\big) \\
    \ll \, & N.
\end{align*}
Back to \eqref{GCDsums}, we have the following lower bound by combining Lemma \ref{lem2.2}
\begin{align*}
    \frac{M_2(R,X)}{M_1(R,X)} \gg \,& x (\log X)^{-\delta}\exp\bigg((2+o(1))\sqrt{\frac{\log N\log_3N}{\log_2N}}\bigg) \\
    \ge \, &x\exp\bigg((2+o(1))\sqrt{\frac{\log (X^{\frac12-\delta}/x)\log_3(X^{\frac12-\delta}/x)}{\log_2(X^{\frac12-\delta}/x)}}\bigg).
\end{align*}
Returning to \eqref{max}, we complete the proof of Theorem \ref{thm1.3}, since $\kappa$ can be arbitrarily small.

	\section*{Acknowledgements}
		Z. Dong is supported by the Shanghai Magnolia Talent Plan Pujiang Project (Grant No. 24PJD140) and the National
	Natural Science Foundation of China (Grant No. 	1240011770). W. Wang is supported by the China Postdoctoral Science Foundation (Grant No. 2024M763477) and the National
	Natural Science Foundation of China (Grant No. 1250012812). H. Zhang is supported by the Fundamental Research Funds for the Central Universities (Grant No. 531118010622), the National
	Natural Science Foundation of China (Grant No. 1240011979) and the Hunan Provincial Natural Science Foundation of China (Grant No. 2024JJ6120).

	\normalem

\end{document}